\documentclass[]{amsproc}

\theoremstyle{definition}

\theoremstyle{remark}

\numberwithin{equation}{section}
\usepackage{graphicx}
\usepackage{latexsym}
\usepackage{amsmath}
\usepackage{amssymb}
\usepackage{amsfonts}
\usepackage{subfigure}
\usepackage{verbatim}
\usepackage{mathrsfs}
\usepackage[modulo]{lineno} % 行数
\usepackage[latin1]{inputenc}
\usepackage{color}
\usepackage[colorlinks,citecolor=blue,urlcolor=blue]{hyperref}

\newtheorem{tm}{Theorem}[section]
\newtheorem{rk}{Remark}[section]
\newtheorem{ap}{Assumption}[section]

\newtheorem{lm}{Lemma}[section]

\newcommand{\ee}{\mathbb E}
\newcommand{\pp}{\mathbb P}
\newcommand{\nn}{\mathbb N}
\newcommand{\zz}{\mathbb Z}
\newcommand{\rr}{\mathbb R}

\newcommand{\CC}{\mathcal C}

\newcommand{\OO}{\mathcal O}
\newcommand{\PP}{\mathcal P}

\newcommand{\OOO}{\mathscr O}
\newcommand{\FFF}{\mathscr F}

\newcommand{\<}{\langle}
\renewcommand{\>}{\rangle}
\allowdisplaybreaks \allowdisplaybreaks[4]

\begin{document}

%\linenumbers

\title[Strong Approximation of Monotone SPDEs with White Noise]
{Strong Approximation of Monotone Stochastic Partial Differential Equations driven by White Noise}

\author{Zhihui Liu}
\address{Department of Mathematics, 
The Hong Kong University of Science and Technology, 
Clear Water Bay, Kowloon, Hong Kong}
\curraddr{}
\email{zhliu@ust.hk}
\thanks{}
%%    author one information
\author{Zhonghua Qiao}
\address{Corresponding author. Department of Applied Mathematics,
The Hong Kong Polytechnic University,
Hung Hom, Kowloon, Hong Kong}
\curraddr{}
\email{zqiao@polyu.edu.hk}
%\thanks{}



\date{\today}

\dedicatory{}

\begin{abstract}
We establish an optimal strong convergence rate of a fully discrete numerical scheme for second order parabolic stochastic partial differential equations with monotone drifts, including the stochastic Allen--Cahn equation, driven by an additive space-time white noise.
Our first step is to transform the original stochastic equation into an equivalent random equation whose solution possesses more regularity than the original one.
Then we use the backward Euler in time and spectral Galerkin in space to fully discretize this random equation.
By the monotone assumption, in combination with the factorization method and stochastic calculus in martingale-type 2 Banach spaces, we derive a uniform maximum norm estimation and a H\"older-type regularity for both stochastic and random equations.
Finally, the strong convergence rate of the proposed fully discrete scheme is obtained.
Several numerical experiments are carried out to verify the theoretical result.
\end{abstract}

\maketitle

%\tableofcontents

\section{Introduction}
\label{sec1}

Strong approximations for stochastic partial differential equations (SPDEs) with Lipschitz coefficients have been well studied, see, e.g., \cite{ACLW16(SINUM), BJK16(SINUM), CHL17(SINUM), CLS13(SINUM)} and references therein.
For certain types of SPDEs driven by colored noises with non-Lipschitz coefficients, \cite{CHL17(JDE), Dor12(SINUM), FLZ17(SINUM)} obtained strong convergence rates for numerical approximations by using the monotonicity or exponential integrability and Sobolev embedding to control the maximum norm bounds of the exact and numerical solutions.
It is an interesting and difficult problem to derive strong convergence rates of fully discrete schemes for second order parabolic SPDEs with non-Lipschitz coefficients driven by space-time white noise.
In particular, to the best of our knowledge, there exist few works on strong approximations of SPDEs with general monotone drifts driven by space-time white noise.
This is the main motivation for the present study.

Our main concern in this paper is to derive the strong convergence rate of a fully discrete scheme for the following parabolic SPDE with monotone drift driven by an additive Brownian sheet $W$ in a stochastic basis $(\Omega,\FFF,(\FFF_t)_{t\in [0,T]},\pp)$:
\begin{align}\label{ac}
\begin{split}
&\frac{\partial u(t,x)}{\partial t}
=\frac{\partial^2 u(t,x)}{\partial x^2} +f(u(t,x))
+\frac{\partial^2 W(t,x)}{\partial t \partial x},
\quad (t,x)\in (0,T]\times (0,1),
\end{split}
\end{align}
with the following initial value and homogeneous Dirichlet boundary condition:
\begin{align}\label{dbc}
u(t,0)=u(t,1)=0, \
u(0,x)=u_0(x),
\ \ (t,x)\in [0,T]\times (0,1).
\end{align}
Here $f$ satisfies certain monotone condition with polynomial growth derivative (see Assumption \ref{ap-f}).
We remark that if $f(x)=x-x^3$, then Eq. \eqref{ac}-\eqref{dbc} is called the stochastic Allen--Cahn equation or the stochastic Ginzburg--Landau equation, which has been extensively studied mathematically and numerically in literature; see, e.g., \cite{FLP14(SPDE), FLZ17(SINUM), Fun16, KLL15(JAP), LQ17-1, Pro(MOC)} and references cited therein.

For a slightly different version of the stochastic Allen--Cahn equation with space-time white noise, \cite[Theorem 3.1]{YZ17(JMAA)} got a convergence rate in probability sense for spectral Galerkin approximations.
The first result on strong approximations of second order SPDEs with monotone drifts driven by space-time white noise is given in \cite[Corollary 6.17]{BJ(SPA)} for SPDEs with polynomial drifts.
There the authors obtained the strong convergence rate for a temporally semidiscrete nonlinearity-truncated, Euler-type scheme.
Their method was then used in \cite{BGJK17} to a nonlinearity-truncated, fully discrete scheme for the stochastic Allen--Cahn equation with space-time white noise.
The authors proved that
\begin{align} \label{err-jen}
\sup_{0\le m\le M}\|u(t_m)-u_N^m\|_{L^2(\Omega\times (0,1))}
=\mathcal O \big(N^{-\beta}+M^{-\beta/2} \big),
\end{align}
for any $\beta\in (0,1/2)$, where $u_N^m$ denotes the numerical solution and $N,M$ are the dimension of spectral Galerkin and the number of temporal steps, respectively.
The authors in \cite{BCH18(IMA)} analyzed the strong convergence rate of a temporal splitting scheme of the stochastic Allen--Cahn equation with space-time white noise based on the explicit solvability of the phase flow of ${\rm d}u/{\rm d}t=(u-u^3)$, and \cite{Wan18} gave sharp strong convergence rate of a tamed fully discrete exponential integrator for SPDE with cubic nonlinearity and negative leading coefficient.

In this work, we consider more general SPDEs with monotone drifts, which include the stochastic Allen--Cahn equation studied in aforementioned references.
Our strong approximation of Eq. \eqref{ac}-\eqref{dbc} consists of two steps.
The first step is to transform the original stochastic equation \eqref{ac} into an equivalent random equation \eqref{z} whose solution possesses more regularity than the original one.
The spatial spectral Galerkin approximation of Eq. \eqref{ac}-\eqref{dbc} is exactly the sum of the spectral Galerkin approximation of the aforementioned random equation \eqref{z} and the spectral approximate Ornstein--Uhlenbeck process; see \eqref{spe-z}.
Then we use the natural backward Euler scheme \eqref{full} to discretize the random spectral Galerkin approximate equation \eqref{spe-z}.
To derive the strong convergence rate of this fully discrete approximation, we make full use of the monotonicity of the random equation, in combination with the factorization method and stochastic calculus in martingale-type 2 Banach spaces, to derive {\it a priori}
maximum norm estimation and a H\"older-type regularity for the solutions of Eq. \eqref{ac}-\eqref{dbc} and \eqref{z} (see Lemmas \ref{reg} and \ref{reg-hol}).
It has been noted that such stochastic-random transformation was used in \cite[Section 7.2]{DZ14} and references cited therein to mathematically analyze SPDEs driven by additive noise.
We believe that this is the first work that uses such strategy to analyze strong convergence rates of numerical schemes for SPDEs.

Our main result shows that the proposed fully discrete scheme possesses the following convergence rate under the $l_t^\infty L^2_\omega L^2_x \cap l_t^q L^q_\omega L^q_x$-norm for certain $q\ge 2$ and for any $\gamma\in (0,1/2)$ (see Theorem \ref{u-unm}):
\begin{align} \label{err-lq}
\sup_{0\le m\le M}\ee\Big[\|u(t_m)-u_N^m\|^2_{L^2(0,1)} \Big]
& +\frac1M\sum_{m=0}^M \ee\Big[\|u(t_m)-u_N^m\|^q_{L^q(0,1)}\Big]
\nonumber \\
& = \OO \big(N^{-2\gamma}+M^{-1/2} \big).
\end{align}
Taking into account of the optimal Sobolev regularity in Lemma \ref{reg} and a reverse estimation \eqref{rk-ou-err0}, the convergence rate \eqref{err-lq} is sharp.
It should be noted that the proposed scheme is implicit which avoids the truncation or tame of the nonlinearity, and its temporal mean-square convergence order is $1/4$ which removes an infinitesimal factor of \eqref{err-jen} appeared in \cite{BGJK17}.

The rest of this article is organized as follows.
Some preliminaries and {\it a priori} maximum norm estimation and a H\"older-type regularity for the solutions of Eq. \eqref{ac}-\eqref{dbc} and \eqref{z} are given in the next section, followed by the strong convergence analysis for the proposed fully discrete scheme in Section \ref{sec3}.
Several numerical experiments are given to support theoretical claims in the last section.

\section{Preliminaries}
\label{sec2}

In this section, we give some commonly used notations and the optimal spatial Sobolev and temporal H\"older regularity for the solution of Eq. \eqref{ac}-\eqref{dbc}.
They are used in the next section to deduce the sharp strong convergence rate of a fully discrete scheme.

\subsection{Notations}

Let $p\ge 1$, $r\in [1,\infty]$, $q\in [2,\infty]$, $\theta\ge 0$ and $\delta\in [0,1]$.
Here and after we denote $L_x^q:=L_x^q(0,1)$ and $H:=L_x^2$ with norm $\|\cdot\|$ and inner product
$\<\cdot, \cdot\>$.
Similarly, $L_\omega^p$ and $L_t^r$ denote the related Lebesgue spaces on the filtered probability space (also called stochastic basis) $(\Omega,\FFF,(\FFF_t)_{t\in [0,T]},\pp)$ and $(0,T)$, respectively.
For convenience, sometimes we use the temporal, sample path and spatial mixed norm $\|\cdot\|_{L_\omega^p L_t^r L_x^q}$ in different orders, such as
\begin{align} \label{norm}
\|u\|_{L_\omega^p L_t^r L_x^q}
:=\bigg(\int_\Omega \bigg(\int_0^T \bigg(\int_0^1 |u(t,x,\omega)|^q {\rm d}x\bigg)^\frac rq {\rm d}t\bigg)^\frac pr {\rm d}\pp(\omega)\bigg)^\frac 1p
\end{align}
for $u\in L_\omega^p L_t^r L_x^q$,
with the usual modification for $r=\infty$ or $q=\infty$.

Denote by $A$ the Dirichlet Laplacian on either $H$ or $L_x^q$.
Then $A$ is the infinitesimal generator of an analytic $C_0$-semigroup $S(\cdot)$ on $H$ or $L_x^q$, and thus one can define the fractional powers $(-A)^\theta$ of the operator $-A$.
Let $\theta\ge 0$ and $W_x^{\theta,q}$ ($\dot H^\theta:=W_x^{\theta,2}$) be the domain of $(-A)^{\theta/2}$ equipped with the norm $\|\cdot\|_{\theta,q}$
($\|\cdot\|_\theta:=\|\cdot\|_{\theta,2}$):
\begin{align*}
\|u\|_{\theta,q}
:=\|(-A)^{\theta/2} u\|_{L_x^q},
\quad u\in W_x^{\theta,q}.
\end{align*}

For a Banach space $(B,\|\cdot\|_B)$ and a bounded closed subset
$\OOO\subset \rr^d$, we use $\CC(\OOO; B)$ to denote the Banach space consisting of $B$-valued continuous functions $f$ such that
$\|f\|_{\CC(\OOO; B)}:=\sup_{x\in \OOO} \|f(x)\|_B<\infty$, and
$\CC^\delta(\OOO; B)$ with $\delta\in (0,1]$ to denote the $B$-valued function $f$ such that
\begin{align*}
\|f\|_{\CC^\delta(\OOO; B)}
:=\sup_{x \in \OOO} \|f(x)\|_B+\sup_{x,y\in \OOO, x\neq y}\frac{\|f(x)-f(y)\|_B}{|x-y|^\delta}<\infty.
\end{align*}
In the following, when $B=\rr$ and $\OOO=[0,1]$ we simply denote
$\CC^\delta([0,1]; \rr)=\CC^\delta$.
Similarly, we use $L^p(\Omega;\CC([0,T]; B))$ to denote the Banach space consisting of $B$-valued a.s. continuous stochastic processes $u=\{u(t):\ t\in [0,T]\}$ such that
\begin{align*}
\|u\|_{L^p(\Omega;\CC([0,T]; B))}
:=\bigg(\ee\bigg[\sup_{t\in [0,T]} \|u(t)\|_B^p \bigg]\bigg)^\frac1p
<\infty,
\end{align*}
and
$L^p(\Omega;\CC^\delta([0,T]; B))$ with $\delta\in (0,1]$ to denote $B$-valued stochastic processes $u=\{u(t):\ t\in [0,T]\}$ such that
\begin{align*}
\|u\|_{L^p(\Omega;\CC^\delta([0,T]; B))}
: & =\bigg(\ee\bigg[\sup_{t\in [0,T]} \|u(t)\|_B^p\bigg]\bigg)^\frac1p \\
&\quad +\bigg(\ee\bigg[\bigg(\sup_{t,s\in [0,T], t\neq s}\frac{\|u(t)-u(r)\|_B}{|t_2-t_1|^\delta}\bigg)^p\bigg]\bigg)^\frac1p<\infty.
\end{align*}

The main condition on the nonlinear function $f$ is the following monotone-type assumption.

\begin{ap} \label{ap-f}
There exist constants $b\in \rr$, $L_f,\widetilde{L_f}>0$ and $q\ge 2$ such that
\begin{align}
(f(x)-f(y)) (x-y) \le b |x-y|^2-L_f |x-y|^q,&
\quad x, y\in \rr;  \label{con-f} \\
 |f(0)|<\infty,\quad
|f'(x)| \le \widetilde{L_f} (1+|x|^{q-2}),&
\quad x \in \rr.  \label{con-f'}
\end{align}
\end{ap}

It is clear from \eqref{con-f'} that $f$ grows at most polynomially of degree
$(q-1)$ by the mean value theorem:
\begin{align} \label{con-f1}
|f(x)| \le C(1+|x|^{q-1}),
\quad x\in \rr,
\end{align}
where $C=C(|f(0)|,\widetilde{L_f})$ is a positive constant.
A motivated example of $f$ such that Assumption \ref{ap-f} holds true is a polynomial of odd degree $(q-1)$ with negative leading coefficient perturbed with a Lipschitz continuous function; see, e.g., \cite[Exmple 7.8]{DZ14}.

In order to apply the theory of stochastic analysis in infinite dimensional settings, we need to transform the original SPDE \eqref{ac} into an infinite dimensional stochastic evolution equation.
To this end, let us define $F: L_x^{q'} \rightarrow L_x^q$ by the Nemytskii operators associated with $f$:
\begin{align*}
F(u)(x):=f(u(x)), \quad x\in [0,1].
\end{align*}
where $q'$ denote the conjugation of $q$, i.e., $1/q'+1/q=1$.
Then by Assumption \ref{ap-f}, the operator $F$ has a continuous extension from $L^{q'}_x$ to $L^q_x$ and satisfies
\begin{align} \label{con-F}
_{L^{q'}_x}\<F(x)-F(y), x-y\>_{L^q_x}
\le b \|x-y\|^2-L_f \|x-y\|^q_{L^q_x},
\quad x,y\in L^q_x,
\end{align}
where $_{L^{q'}_x} \<\cdot, \cdot\>_{L^q_x}$ denotes the dual between $L^{q'}_x$ and $L^q_x$.

Denote by $W_H$ the $H$-valued cylindrical Wiener process in the stochastic basis 
$(\Omega,\FFF,(\FFF_t)_{t\in [0,T]},\pp)$, i.e., there exists an orthonormal basis $\{h_k\}_{k=1}^\infty $ of $H$ and a sequence of mutually independent Brownian motions $\{\beta_k\}_{k=1}^\infty $ such that
\begin{align}\label{wiener}
W_H(t)=\sum_{k=1}^\infty h_k\beta_k(t),\quad t\in [0,T].
\end{align}
Then Eq. \eqref{ac}-\eqref{dbc} is equivalent to the following stochastic evolution equation:
\begin{align}\label{AC} \tag{SACE}
{\rm d}u(t)=(Au(t)+F(u(t))) {\rm d}t+{\rm d}W_H(t),
\ t\in (0,T];
\quad u(0)=u_0.
\end{align}

Note that for any $q\ge 2$ and $\theta\ge 0$, the function space $W_x^{\theta,q}$ is a martingale-type $2$ Banach space.
We need the following Burkholder inequality in martingale-type $2$ Banach space (see, e.g., \cite[Theorem 2.4]{Brz97(SSR)}):
\begin{align}\label{bdg}
\bigg\|\int_{0}^t \Phi(r){\rm d}W_H(r)\bigg\|_{L_\omega^p L_t^\infty L_x^q}
\le C \big\|\Phi \big\|_{L^p(\Omega; L^2(0,T; \gamma(H,L_x^q)))},
\end{align}
for $p,q\ge 2$, where $\gamma(H,L_x^q)$ denotes the radonifying operator norm:
\begin{align*}
\|\Phi\|_{\gamma(H,L_x^q)}
:=\bigg\|\sum_{k=1}^\infty\gamma_k \Phi h_k\bigg\|_{L^2(\Omega';L_x^q)}.
\end{align*}
Here $\{h_k\}_{k=1}^\infty$ is any orthonormal basis of $H$ and
$\{\gamma_n\}_{n\geq 1}$ is a sequence of independent
$\mathcal N(0,1)$-random variables on a probability space $(\Omega',\FFF',\pp')$, provided that the above series converges.
We also note that $L_x^q$ with $q\ge 2$ is a Banach function space with finite cotype, and then $\Phi\in \gamma(H;L_x^q)$ if and only if $(\sum_{k=1}^\infty (\Phi h_k)^2)^{1/2}$ belongs to $L_x^q$ for any orthonormal basis $\{h_k\}_{k=1}^\infty$ of $H$; see \cite[Lemma 2.1]{NVW08(JFA)}.
Moreover, in this situation,
\begin{align}\label{cotype}
\|\Phi\|^2_{\gamma(H;L_x^q)}
&\simeq \bigg\|\sum_{k=1}^\infty (\Phi h_k)^2\bigg\|_{L_x^{q/2}},
\quad \Phi\in \gamma(H;L_x^q).
\end{align}

For convenience, we frequently use the generic constant $C$, which may be different in each appearance and is independent of the discrete parameters $N$ and $M$ or equivalently, $\tau$, respectively.

\subsection{{\it A Priori} Estimation}

Recall that a predictable stochastic process $u:[0,T]\times \Omega\rightarrow H$  is called a mild solution of Eq. \eqref{AC} if $u\in L^\infty(0,T; H)$ a.s. such that
\begin{align} \label{mild}
u(t)=S(t)u_0+\int_0^t S(t-r) F(u(r)){\rm d}r+W_A(t), \quad a.s.
\quad t\in [0,T],
\end{align}
where $S=\{S(t):=e^{A t}:\ t\in [0,T]\}$ is the analytic $\CC_0$-semigroup generalized by $A$,
$$S*F(u)=\bigg\{\int_0^t S(t-r) F(u(r)){\rm d}r:\ t\in [0,T]\bigg\}$$
is the deterministic convolution and
$$W_A=\bigg\{W_A(t)=\int_0^t S(t-r) {\rm d}W_H(r):\ t\in [0,T]\bigg\}$$
is the so-called Ornstein--Uhlenbeck process.
The uniqueness of the mild solution of Eq. \eqref{AC} is understood in the sense of stochastic equivalence.

Set $z(t):=u(t)-W_A(t)$, $t\in [0,T]$.
Then it is clear that $u$ is the unique solution of Eq. \eqref{AC} if and only if $z$ is the unique mild solution of the following random partial differential equation:
\begin{align}\label{z}
\dot z(t)=A z(t)+F(z(t)+W_A(t)), \quad t\in [0,T]; \quad
z(0)=u_0.
\end{align}
The mild solution of the above equation is equivalent to its variational solution (see, e.g., \cite[Theorem 5.4]{DZ14}), i.e.,
for any subdivision $\{0=t_0<t_1<\cdots<t_m<t_{m+1}<\cdots<t_M=T\}$ with $M\in \nn_+$ of the time interval $[0,T]$ and $v\in \dot H^1$ it holds a.s. that
\begin{align}\label{z-var}
\<z(t_{m+1})-z(t_m), v\>
+\int_{t_m}^{t_{m+1}} \<\nabla z, \nabla v\> {\rm d}r
=\int_{t_m}^{t_{m+1}} \<F(u), v\> {\rm d}r,
\end{align}
for any $m\in \zz_{M-1}:=\{0,1,\cdots,M-1\}$.

The existence of a unique mild solution of Eq. \eqref{mild} under the monotone condition \eqref{con-F}, and thus Eq. \eqref{ac}-\eqref{dbc} under Assumption \ref{ap-f}, had been established in \cite[Theorem 7.17]{DZ14}.
We will give a uniform moments' estimation of this solution in Lemma \ref{reg} with the aforementioned monotone condition \eqref{con-F} following some ideas of \cite[Proposition 2.1]{LQ17-1}.
For simplicity, we assume that the initial datum $u_0$ is a deterministic function; the case of random $u_0$ possessing certain bounded $p$-moments can also be handled by similar arguments as in \cite[Proposition 2.1]{LQ17-1}.

As in \cite[Lemma 2.1]{LQ17-1} where we have shown that the Sobolev and H\"older regularity of the Ornstein--Uhlenbeck process $W_A$, our main tool is the following factorization formula which is valid by deterministic and stochastic Fubini theorems:
\begin{align*}
S*F(u)(t)
&=\int_0^t S(t-r) F(u(r)) {\rm d}r
=\frac{\sin(\pi \alpha)}{\pi} \int_0^t (t-r)^{\alpha-1} S(t-r) F_\alpha(r) {\rm d}r, \\
W_A(t)
&=\int_0^t S(t-r) {\rm d}W_H(r)
=\frac{\sin(\pi \alpha)}{\pi} \int_0^t (t-r)^{\alpha-1} S(t-r) W_\alpha(r) {\rm d}r,
\end{align*}
where $\alpha\in (0,1)$ and
\begin{align*}
F_\alpha(t)
:&=\int_0^t (t-r)^{-\alpha} S(t-r) F(u(r)) {\rm d}r, \\
W_\alpha(t)
:&=\int_0^t (t-r)^{-\alpha} S(t-r) {\rm d}W_H(r),
\quad t\in [0,T].
\end{align*}
It was proved in \cite[Lemma 3.3]{Brz97(SSR)} that, when $p>1$ and $1/p<\alpha<1$, the linear operator $R_\alpha$ defined by
\begin{align*}
R_\alpha f(t)
:=\int_0^t (t-r)^{\alpha-1} S(t-r) f(r) {\rm d} r,\quad t\in [0,T],
\end{align*}
is bounded from $L^p(0,T; L_x^q)$ to $\CC^{\delta}([0,T]; W_x^{\theta,q})$ with $\delta<\alpha-1/p$ when $\theta=0$ or $\delta=\alpha-1/p-\theta/2$ when
$\theta>0$ and $\alpha>\theta/2+1/p$.

\begin{lm} \label{reg}
Let $\beta\in (0,1/2)$.
Assume that $u_0\in \dot H^\beta \cap L_x^\infty$.
Then for any $p\ge 1$, there exists a constant $C=C(T,p,b,q,L_f,\beta)$ such that
\begin{align} \label{bou}
& \|u\|_{L_\omega^p L_t^\infty L_x^\infty}
+\|u\|_{L_\omega^p L_t^\infty \dot H^\beta}
+\|z\|_{L_\omega^p L_t^\infty L_x^\infty}
+\|z\|_{L_\omega^p L_t^\infty \dot H^\beta} \nonumber \\
& \le C\Big(1+\|u_0\|_{L_x^\infty}^{q-1}+\|u_0\|^{q-1}_{\dot H^\beta} \Big),
\end{align}
and that
\begin{align}\label{hol}
\|u(t)-u(s)\|_{L^p(\Omega; H)}
\le C|t-s|^{\beta/2}, \quad t,s\in [0,T].
\end{align}
Moreover, if $u_0\in \dot H^{1/2} \cap L_x^\infty$.
Then
\begin{align}\label{hol+}
\|u(t)-u(s)\|_{L^p(\Omega; H)}
\le C|t-s|^{1/4}, \quad t,s\in [0,T].
\end{align}
\end{lm}

\begin{proof}
For the initial term in Eq. \eqref{mild}, by the property of the semigroup $S$,
\begin{align}
\|S(t)u_0\|_{L_x^\infty}+\|S(t)u_0\|_{\beta}
&\le C(\|u_0\|_{L_x^\infty}+\|u_0\|_{\beta}), \label{est-x0} \\
\|S(t)u_0-S(s)u_0\|
&\le C |t-s|^{\beta/2} \|u_0\|_{\beta}. \label{hol-x0}
\end{align}
Let $p,q\ge 2$ and $t\in (0,T]$.
Applying Fubini theorem and the Burkholder inequality \eqref{bdg}, we have
\begin{align*}
\big\| W_\alpha \big\|^p_{L_\omega^p L_t^p L_x^q}
&=\int_0^T \ee\bigg[\bigg\| \int_0^t (t-r)^{-\alpha} S(t-r) {\rm d}W_H(r)\bigg\|_{L_x^q} ^p \bigg]  {\rm d}t  \\
&\le C \int_0^T \bigg( \int_0^t r^{-2\alpha}
\|S(r)\|^2_{\gamma(H;L_x^q)} {\rm d} r \bigg)^\frac p2 {\rm d}t.
\end{align*}
Then by \eqref{cotype} and the uniform boundedness of $\{e_k=\sqrt 2\sin(k\pi\cdot)\}_{k=1}^\infty$, we get
\begin{align*}
\|S(t)\|^2_{\gamma(H;L_x^q)}
&\simeq \bigg\|\sum_{k=1}^\infty (S(t) e_k)^2\bigg\|_{L_x^{q/2}}
\le \sum_{k=1}^\infty e^{-2\lambda_k t} \|e_k\|^2_{L_x^q}
\le C t^{-\frac12},
\end{align*}
where the elementary inequality
$\sum_{k=1}^\infty e^{-2\lambda_k t}
\le C t^{-\frac12}$ is used.
Then
\begin{align*}
\big\| W_\alpha \big\|_{L_\omega^p L_t^p  L_x^q}
&\le C \bigg(\int_0^T \bigg( \int_0^t r^{-(2\alpha+\frac12)} {\rm d} r \bigg)^\frac p2 {\rm d}t\bigg)^\frac1p,
\end{align*}
which is finite if and only if $\alpha\in (0,1/4)$.
As a result of the H\"older continuity characterization, $W_A \in L^p(\Omega; \CC^{\delta}([0,T]; W_x^{\theta,q}))$ for any $\delta,\theta\ge 0$ with
$\delta+\theta/2<1/4$.
By the Sobolev embedding
$W_x^{\theta,q} \hookrightarrow L_x^\infty
\cap \dot H^\beta$ with sufficiently large $q$ and $\beta\le \theta<1/2$, we conclude that
\begin{align}
& \ee\Big[\sup_{t\in [0,T]} \| W_A(t) \|^p_{L_x^\infty}\Big]
+\ee\Big[\sup_{t\in [0,T]} \| W_A(t) \|^p_{\beta}\Big]
\le C, \label{reg-w} \\
& \|W_A(t)-W_A(s)\|_{L^p(\Omega; H)}
\le C|t-s|^\gamma, \quad t,s\in [0,T],   \label{hol-wa}
\end{align}
for any $p\ge 1$, $\beta\in (0,1/2)$ and $\gamma\in (0,1/4)$.

In terms of \eqref{reg-w} and the relation $z=u-W_A$, to show the estimation \eqref{bou} for $u$ and $z$ it suffices to show one of them.
Let $L\ge 1$.
Testing both sides of Eq. \eqref{z} by $|z|^{2(L-1)} z$ and integrating by parts yield that
\begin{align*}
& \frac1{2L} \|z(t)\|^{2L}_{L^{2L}_x}
+(2L-1) \int_0^t \< |z(r)|^{2(L-1)}, |\nabla z(r)|^2 \> {\rm d}r   \\
&=\frac1{2L} \|u_0\|^{2L}_{L^{2L}_x}
+\int_0^t \<(F(u(r)), |z(r)|^{2(L-1)} z(r)\> {\rm d}r.
\end{align*}
It follows from the condition \eqref{con-f} and Young inequality that
\begin{align*}
& \int_0^t \<(F(u(r)), |z(r)|^{2(L-1)} z(r) \> {\rm d}r \\
&=\int_0^t \<F(z(r)+W_A(r))-F(W_A(r)), z^{2L-1}(r)\> {\rm d}r
-\int_0^t \< W_A(r), z^{2L-1}(r)\> {\rm d}r \\
&\le C \int_0^t \|z(r)\|^{2L}_{L^{2L}_x} {\rm d}r
-L_f \int_0^t \|u(r)\|^{q+2(L-1)}_{L^{q+2(L-1)}_x}  {\rm d}r
+C \int_0^t \|W_A(r)\|^{2L}_{L^{2L}_x} {\rm d}r.
\end{align*}
Thus we obtain
\begin{align*}
& \frac1{2L} \|z(t)\|^{2L}_{L^{2L}_x}
+L_f \int_0^t \|u(r)\|^{q+2(L-1)}_{L^{q+2(L-1)}_x} {\rm d}r \\
&\le \frac1{2L} \|u_0\|^{2L}_{L^{2L}_x}
+C \int_0^t \|z(r)\|^{2L}_{L^{2L}_x} {\rm d}r
+C \int_0^t \|W_A(r)\|^{2L}_{L^{2L}_x} {\rm d}r.
\end{align*}
Now taking $L^1_\omega L^\infty_t$-norm, we conclude from Gr\"onwall inequality and \eqref{reg-w} that
\begin{align*}
\ee\Big[\sup_{t\in [0,T]} \|z(t)\|^{2L}_{L^{2L}_x} \Big]
+\int_0^T \ee\Big[\|z(r)\|^{q+2(L-1)}_{L^{q+2(L-1)}_x}\Big] {\rm d}t
\le C \Big(1+\|u_0\|^{2L}_{L^{2L}_x} \Big).
\end{align*}
Similarly, one gets by taking $L^{p/2}_\omega L^\infty_t$-norm with general $p\ge 2$ and the relation $z=u-W_A$ that
\begin{align} \label{mom}
\ee\Big[\sup_{t\in [0,T]} \|u(t)\|^p_{L^{2L}_x} \Big]
+\ee\Big[\sup_{t\in [0,T]} \|z(t)\|^p_{L^{2L}_x} \Big]
\le C \Big(1+\|u_0\|^p_{L^{2L}_x}  \Big),
\quad p\ge 2.
\end{align}

Consequently, for any $\alpha\in (0,1)$ we get
\begin{align*}
\|F_\alpha\|^p_{L_\omega^p L_t^p L_x^{2L}}
&\le \int_0^T \ee\bigg[ \bigg(\int_0^t (t-r)^{-\alpha}
\|S(t-r) F(u(r))\|_{L_x^{2L}} {\rm d} r\bigg)^p \bigg] {\rm d}t \\
&\le C\Big( 1+\|u\|^{p(q-1)}_{L_\omega^{p (q-1)} L_t^\infty L_x^{2L(q-1)}}\Big)
\le C \Big( 1+\|u_0\|^{p(q-1)}_{L_x^{2L(q-1)}}\Big).
\end{align*}
Therefore, $S*F(u)\in L^p(\Omega; \CC^{\delta}([0,T]; W_x^{\theta,2L}))$ for any $\delta,\theta\ge 0$ with
$\delta+\theta/2<1$.
Then by Sobolev embedding we have
\begin{align} \label{est-f1}
&\ee\Big[\sup_{t\in [0,T]} \|S*F(u)(t) \|^p_{L_x^\infty}\Big]
+ \ee\Big[\sup_{t\in [0,T]} \|S*F(u)(t)\|^p_{\beta}\Big]
\le C,
\end{align}
and
\begin{align}
& \big\|S*F(u)(t)-S*F(u)(s) \big\|_{L^p(\Omega; H)}
\le C|t-s|^\gamma, \quad t,s\in [0,T],   \label{hol-f1}
\end{align}
for any $p\ge 1$, $\beta\in (0,1/2)$ and $\gamma\in (0,1)$.

Combining \eqref{est-x0}-\eqref{est-f1} and the relation that $u=z+W_A$, we get the estimations \eqref{bou} and \eqref{hol}.
To show the last inequality \eqref{hol+}, we only need to give a refined estimation of \eqref{hol-wa}:
\begin{align} \label{hol-wa+}
\|W_A(t)-W_A(s)\|_{L^p(\Omega;H)}
\le C|t-s|^{1/2}, \quad t,s\in [0,T].
\end{align}
Due to the fact that $W_A$ is Gaussian, we only need to show \eqref{hol-wa+} for $p=2$.
Without loss of generality, assume that $0\le s\le t\le T$.
By It\^o isometry, we have
\begin{align*}
& \ee\Big[\|W_A(t)-W_A(s)\|^2\Big]  \\
&=\ee\bigg[\bigg\|\int_s^t S(t-r) {\rm d}W_H(r)\bigg\|^2\bigg]
+\ee\bigg[\bigg\|\int_0^s (S(t-r)-S(s-r)) {\rm d}W_H(r)\bigg\|^2\bigg] \\
&=\int_0^{t-s} \bigg[\sum_{k=1}^\infty e^{-2\lambda_k r}\bigg] {\rm d}r
+\sum_{k=1}^\infty \frac{1-e^{-2\lambda_k s}}{2\lambda_k}
\Big(1-e^{-\lambda_k(t-s)}\Big)^2 \\
&\le \int_0^{t-s} \bigg[\sum_{k=1}^\infty e^{-2\lambda_k r}\bigg] {\rm d}r
+\frac12 \sum_{k=1}^\infty \frac{1-e^{-\lambda_k(t-s)}}{\lambda_k} \\
&= \int_0^{t-s} \bigg[\sum_{k=1}^\infty e^{-2\lambda_k r}\bigg] {\rm d}r
+\frac12 \int_0^{t-s} \bigg[\sum_{k=1}^\infty e^{-\lambda_k r}\bigg]  {\rm d}r
\le C(t-s).
\end{align*}
This completes the proof of \eqref{hol-wa+}.
\end{proof}

Next, we use the uniform estimation in Lemma \ref{reg} to derive the following H\"older-type regularity of the solutions $u$ and $z$ of Eq. \eqref{ac}-\eqref{dbc} and \eqref{z}, respectively.

\begin{lm} \label{reg-hol}
Let $\beta\in (0,1/2]$.
Assume that $u_0\in \dot H^\beta\cap L_x^\infty$.
Then for any $p\ge 1$, there exists a constant $C=C(T,p,b,L_f,\beta,u_0)$ such that for any $0\le s\le t\le T$ there holds that
\begin{align} \label{hol-f}
\ee\Big[\|F(u(t))-F(u(s)\|^2\Big]
\le C (t-s)^\beta.
\end{align}
Moreover, if $u_0\in \dot H^{1+\beta}$, then
\begin{align}\label{hol-y}
\ee\Big[\|\nabla z(t)-\nabla z(s)\|^2\Big]
\le C (t-s)^\beta.
\end{align}
\end{lm}

\begin{proof}
We start with the first estimation \eqref{hol-f}.
By the mean value theorem, the condition \eqref{con-f'}, the moments' estimation \eqref{bou} and H\"older-type regularity \eqref{hol}-\eqref{hol+} of $u$, we get
\begin{align*}
& \ee\Big[\|F(u(t))-F(u(s)\|^2\Big] \\
& \le C \bigg(1+\sup_{t\in [0,T]} \|u(t) \|^{2(q-2)}_{L_\omega^{2(q-1)} L_x^\infty} \bigg) \times
 \|u(t)-u(s) \|^2_{L_\omega^{2(q-1)} H}   \\
&\le C (t-s)^\beta,
\end{align*}
which proves \eqref{hol-f}.

Next, we prove the last inequality \eqref{hol-y}.
By the smoothness property of the semigroup $S$ and the regularity of $u_0$, we get
\begin{align} \label{reg-x0}
\|S(t)u_0-S(r)u_0\|_1
\le C \|u_0\|_{1+\beta} (t-s)^{\beta/2},
\quad u_0\in \dot H^{1+\beta}.
\end{align}
By Minkovskii inequality, the condition \eqref{con-f'} and the moments' estimation \eqref{bou} of $u$, we obtain
\begin{align*}
& \|S*F(u(t))-S*F(u(r)) \|_{L^2(\Omega; \dot H^1)} \\
&\le \int_s^t \|S(t-r) F(u(r)) \|_{L^2(\Omega; \dot H^1)} {\rm d}r \\
&\quad +\int_0^s \|(S(t-s)-{\rm Id}_H) S(s-r) F(u(r)) \|_{L^2(\Omega; \dot H^1)} {\rm d}r \\
&\le C \sup_{t\in [0,T]}\|F(u(t)) \|_{L^2(\Omega; H)} \times \bigg(\int_s^t (t-r)^{-\frac12} {\rm d}r+(t-s)^\frac12 \bigg)
\le C(t-s)^\frac12,
\end{align*}
Combining the above two estimations, we get \eqref{hol-y}.
\end{proof}

\section{Fully Discrete Approximation}
\label{sec3}

In this section, we study a fully discrete scheme of Eq. \eqref{z} and derive its optimal strong convergence rate.

\subsection{Backward Euler--Spectral Galerkin Approximation}

Let $M,N\in \nn_+$.
Denote by $\PP_N$ the orthogonal projection operator from $H$ to its finite dimensional subspace $V_N$ spanned by the eigenvectors $\{e_k=\sqrt 2\sin(k\pi \cdot)\}_{k=1}^N$ corresponding to the first $N$ eigenvalues $\{\lambda_k=(k\pi)^2\}_{k=1}^N$ of negative Dirichlet Laplacian $-A$:
\begin{align}
\<\PP_N u, v_N\> &=\<u, v_N\>,
\quad u\in H, v_N\in V_N.
\end{align}
Denote by $A_N$ the restriction of the Laplacian operator $A$ on $V_N$.
Then the spectral approximation of Eq. \eqref{ac}-\eqref{dbc} is to find an
$\FFF_t$-adapted $V_N$-valued process
$u_N=\{u_N(t):\ t\in [0,T]\}$ such that
\begin{align} \label{fem}
&{\rm d}u_N(t)=(A_N u_N(t)+\PP_N F(u_N(t))) {\rm d}t
+\PP_N {\rm d}W_H(t),\ t\in [0,T]; \quad
u_N(0)=\PP_N u_0.
\end{align}

The mild solution of Eq. \eqref{fem} is given by
\begin{align*}
u_N(t)=S_N(t)\PP_N u_0+\int_0^t S_N(t-r)\PP_N F(u_N(r)){\rm d}r+W_A^N(t),
\quad t\in [0,T],
\end{align*}
where $S_N=\{S_N(t):=e^{A_N t}:\ t\in [0,T]\}$ is the analytic $\CC_0$-semigroup generated by $A_N$ and
$W_A^N=\{W_A^N(t)=\int_0^t S_N(t-r) \PP_N {\rm d}W_H(r):\ t\in [0,T]\}$ is the approximate Ornstein--Uhlenbeck process.
Define $z_N=u_N-W_A^N$.
Then $z_N$ solves the following random partial differential equation:
\begin{align}\label{spe-z}
\begin{split}
\dot z_N(t)=A_N z_N(t)+\PP_N F(z_N(t)+W_A^N(t)), \ t\in [0,T];  \quad
z_N(0)=\PP_N u_0.
\end{split}
\end{align}
Let $M\in \nn_+$ and denote $\zz_M:=\{0,1,\cdots,M\}$.
Similarly to Eq. \eqref{z-var}, it is clear that the spectral Galerkin approximation \eqref{fem} of Eq. \eqref{ac}-\eqref{dbc} is equivalent to find a $V_N$-valued process $u_N=z_N+W_A^N$ such that for all subdivision $\{t_m:\ m\in \zz_M\}$ of  $[0,T]$ and $v_N\in V_N$ it holds a.s. that
\begin{align}\label{spe-var}
\<z_N(t_{m+1})-z_N(t_m), v_N\>
+\int_{t_m}^{t_{m+1}} \<\nabla z_N, \nabla v_N\> {\rm d}r
=\int_{t_m}^{t_{m+1}} \<F(u_N), v_N\> {\rm d}r.
\end{align}

The backward Euler approximation of Eq. \eqref{spe-var} is to find a $V_N$-valued discrete process $\{z_N^m:\ N\in \nn_+,\ m\in \zz_M\}$ such that for all $v_N\in V_N$ it holds a.s. that
\begin{align}\label{full}
\<z_N^{m+1}-z_N^m, v_N\>+\tau \<\nabla z_N^{m+1}, \nabla v_N\>
=\tau \<F_N^{m+1}, v_N\>,
\end{align}
where $F_N^{m+1}:=F(z_N^{m+1}+W_A^N(t_{m+1}))$, $m\in \zz_{M-1}$.
We call the fully discrete scheme \eqref{full} the backward Euler--spectral Galerkin scheme of Eq. \eqref{z}.
Set
\begin{align} \label{unm}
u_N^m=z_N^m+W_A^N(t_m), \quad m\in \zz_M.
\end{align}
Then $u_N^m$ is an approximation of the solution $u$ of Eq. \eqref{ac}-\eqref{dbc} at $t_m$, $m\in \zz_{M-1}$.
In this sense, \eqref{full}--\eqref{unm} can be seen as the backward Euler--Galerkin scheme of Eq. \eqref{ac}-\eqref{dbc}.
For simplicity, throughout this section we assume that
$\{I_m:=(t_m,t_{m+1}]:\ m\in \zz_{M-1}\}$ is an equal length subdivision of $(0,T]$ and denote by $\tau=t_{m+1}-t_m$, $m\in \zz_{M-1}$, the temporal step size of this subdivision.

\subsection{Strong Convergence Rate}

This section is devoted to establishing the strong convergence rate for the backward Euler--spectral Galerkin scheme \eqref{full}-\eqref{unm} of Eq. \eqref{ac}-\eqref{dbc}.

We begin with the following essentially optimal error estimation between the Ornstein--Uhlenbeck process $W_A$ and its approximation $W_A^N$, as well as a uniform $L_x^\infty$-bound for $\PP_N u$ with respect to $N$.

\begin{lm} \label{ou-err}
Let $p\ge 1$.
There exists a constant $C=C(p)$ such that
\begin{align} \label{ou-err0}
\sup_{t\in [0,T]}\Big(\ee\Big[\|W_A(t)-W_A^N(t)\|^p\Big]\Big)^\frac1p
\le C N^{-\frac12}.
\end{align}
\end{lm}

\begin{proof}
The difference of the Ornstein--Uhlenbeck processes can be rewritten as
\begin{align*}
W_A(t)-W_A^N(t)
=\int_0^t (S(t-r)-S_N(t-r)\PP_N) {\rm d}W_H(r),
\quad t\in [0,T].
\end{align*}
Since $W_A-W_A^N$ is Gaussian, we only need to show \eqref{ou-err0} for $p=2$.
By It\^o isometry and elementary calculations, we get
\begin{align*}
\sup_{t\in [0,T]} \ee\Big[ \|W_A(t)-W_A^N(t)\|^2 \Big]
&\le \sum_{k=1}^\infty
\int_0^T \|(S(r)-S_N(r)\PP_N)e_k\|^2 {\rm d}r \\
&=\sum_{k=N+1}^\infty \frac{1-e^{-2\lambda_k T}}{2\lambda_k}
\le \frac1{2\pi^2} N^{-1}.
\end{align*}
This completes the proof of \eqref{ou-err0}.
\end{proof}

\begin{rk} \label{rk-ou-err}
The estimation \eqref{ou-err0} is sharp in the sense that
\begin{align} \label{rk-ou-err0}
\ee\Big[ \|W_A(t)-W_A^N(t)\|^2 \Big]
=\sum_{k=N+1}^\infty \frac{1-e^{-2\lambda_k t}}{2\lambda_k}
\ge \frac t{2(1+2\pi^2 t)} N^{-1},
\end{align}
for $t>0$, where we have used the elementary estimation $e^x\ge 1+x$ for any $x\ge 0$.
\end{rk}

\begin{lm}
Let $\epsilon>0$ and $u_0\in \dot H^{\frac12+\epsilon}$.
Then for any $p\ge 1$, there exists a constant $C=C(T,p,\epsilon)$ such that
\begin{align} \label{bou-pnx}
\sup_{N\in \nn_+}\sup_{t\in [0,T]}\|\PP_N u(t)\|_{L^p(\Omega; L_x^\infty)}
\le C\Big(1+\|u_0\|_{L_x^\infty}^{q-1} \Big).
\end{align}
\end{lm}

\begin{proof}
It is clear that
$$\PP_N u(t)=S(t) \PP_N u_0+\PP_N \bigg[\int_0^t S(t-r) F(u(r)) {\rm d}r\bigg]+W_A^N(t).$$
By stochastic Fubini theorem, the approximate Ornstein--Uhlenbeck process $W_A^N$ possesses the following factorization formula:
\begin{align*}
\int_0^t S_N(t-r)\PP_N {\rm d}W_H(r)
=\frac{\sin(\pi \alpha)}{\pi} \int_0^t (t-r)^{\alpha-1} S_N(t-r) W_\alpha^N(t) {\rm d}r,
\end{align*}
where $\alpha\in (0,1)$ and
$W_\alpha^N(t):=\int_0^t (t-r)^{-\alpha} S_N(t-r)\PP_N {\rm d}W_H(r)$, $t\in [0,T]$.
Let $p,q\ge 2$ and $t\in (0,T]$.
Applying Fubini theorem and the Burkholder inequality \eqref{bdg} as well as the equivalence \eqref{cotype} of $\gamma$-norm, we get similarly to Lemma \ref{reg} that
\begin{align*}
\big\| W_\alpha^N \big\|^p_{L_\omega^p L_t^p L_x^q}
&=\int_0^T \ee\bigg[\bigg\| \int_0^t (t-r)^{-\alpha} S_N(t-\tau) \PP_N {\rm d}W_H(r) \bigg\|_{L_x^q}^p \bigg]  {\rm d}t  \\
&\le C \int_0^T \bigg( \int_0^t r^{-2\alpha}
\|S_N(r) \PP_N\|^2_{\gamma(H;L_x^q)} {\rm d} r \bigg)^\frac p2 {\rm d}t \\
&\le C \int_0^T \bigg( \int_0^t r^{-2\alpha}
\bigg\|\sum_{k=1}^\infty (S_N(r) \PP_N e_k)^2\bigg\|_{L_x^{q/2}} {\rm d} r \bigg)^\frac p2 {\rm d}t\\
&\le C \int_0^T \bigg( \int_0^t r^{-(2\alpha+\frac12)} {\rm d} r \bigg)^\frac p2 {\rm d}t.
\end{align*}
The last integral is finite if and only if $\alpha\in (0,1/4)$.
As a result of the H\"older continuity characterization and  Sobolev embedding, $W_A^N \in L^p(\Omega; \CC^{\delta}([0,T]; \CC^\kappa))$ for any $\delta,\kappa\ge 0$ with $\delta+\kappa/2<1/4$ uniformly with respect to $N$.
In particular, there exists a constant $C=C(T,p)$ such that
\begin{align}\label{ou-N-infty}
\sup_{N\in \nn_+} \ee\Big[\sup_{t\in [0,T]} \|W_A^N\|^p_{L_x^\infty}\Big] \le C.
\end{align}
It is shown in Lemma \ref{reg} that
$\int_0^\cdot S(\cdot-r) F(u(r)) {\rm d}r\in L^p(\Omega; \CC^{\delta}([0,T]; W_x^{\theta,2L}))$ for any $\delta,\theta\ge 0$ with
$\delta+\theta/2<1$.
In particular, $\int_0^\cdot S(\cdot-r) F(u(r)) {\rm d}r\in L^p(\Omega; \CC([0,T]; \dot H^\gamma))$ for any $p\ge 1$ and $\gamma\in (0,2)$.
Therefore, by the Sobolev embedding $\dot H^{1/2+\epsilon}\subset L_x^\infty$ there exists a constant $C=C(T,p,\epsilon,u_0)$ such that
\begin{align*}
& \sup_{N\in \nn_+} \ee\Big[\sup_{t\in [0,T]}
\|\PP_N [S*F(u)(t)]\|^p_{L_x^\infty}\Big] \\
& \le C \sup_{N\in \nn_+} \ee\Big[\sup_{t\in [0,T]}
\|\PP_N [ S*F(u)(t) ]\|^p_{\frac12+\epsilon} \Big] \\
& \le C \sup_{N\in \nn_+} \ee\Big[\sup_{t\in [0,T]}
\|S*F(u)(t)\|^p_{\frac12+\epsilon} \Big]
\le C.
\end{align*}
Similarly,
\begin{align*}
\sup_{N\in \nn_+} \ee\bigg[\sup_{t\in [0,T]} \|S(t) \PP_N u_0\|^p_{L_x^\infty}\bigg]
\le C \|u_0\|^p_{\frac12+\epsilon}.
\end{align*}
Therefore, \eqref{bou-pnx} holds.
\end{proof}

Now we can give and prove our main result on convergence rate of the backward Euler--spectral Galerkin scheme \eqref{full}-\eqref{unm} under the $l_t^\infty L^2_\omega L^2_x \cap l_t^q L^q_\omega L^q_x$-norm for Eq. \eqref{ac}-\eqref{dbc}.
Here the $l_t^\infty L^2_\omega L^2_x$-norm and $l_t^q L^q_\omega L^q_x$-norm are temporally discrete norms similarly to the continuous norm given in \eqref{norm}.

\begin{tm} \label{u-unm}
Let $\tau\in (0,1)$ when $b<0$ and $\tau<1/(4b)$ when $b>0$.
Assume that $u_0\in \dot H^{3/2}$.
Let $u$ and $u_N^m$ denote the solutions of Eq. \eqref{AC} and the scheme \eqref{full}-\eqref{unm}, respectively.
Then for any $\gamma\in (0,1/2)$, there exists a constant
$C=C(T,b,L_f,\gamma,\|u_0\|_{3/2})$ such that
\begin{align} \label{u-unm0}
\sup_{m\in \zz_M}\ee\Big[\|u(t_m)-u_N^m\|^2 \Big]
+\sum_{m\in \zz_M} \ee\Big[\|u(t_m)-u_N^m\|_{L_x^q}^q \Big] \tau
\le C \big(N^{-2\gamma}+\tau^{1/2} \big).
\end{align}
\end{tm}

\begin{proof}
Let $\gamma\in (0,1/2)$.
Define $e_N^m:=\PP_N z(t_m)-z_N^m$, $m\in \zz_M$.
Then noting the relation between $u$ and $z$, we get $e_N^m\in V_N$ and
$$u(t_m)-u_N^m=({\rm Id}_H-\PP_N) u(t_m)+e_N^m,
\quad m\in \zz_M.$$
By triangle inequality and the moment's estimation \eqref{bou}, we get
\begin{align} \label{x-xnm01}
& \sup_{m\in \zz_M}\ee\Big[\|u(t_m)-u_N^m\|^2 \Big]
+\sum_{m\in \zz_M} \ee\Big[\|u(t_m)-u_N^m\|_{L_x^q}^q \Big] \tau
\nonumber \\
&\le \sup_{m\in \zz_M} \ee\Big[ \|({\rm Id}_H-\PP_N) u(t_m)\|^2\Big]
+\sum_{m\in \zz_M} \ee\Big[\|({\rm Id}_H-\PP_N) u(t_m)\|_{L_x^q}^q \Big] \tau
\nonumber  \\
&\quad +\sup_{m\in \zz_M} \ee\Big[\|e_N^m\|^2 \Big]
+\sum_{m\in \zz_M} \ee\Big[\|e_N^m\|_{L_x^q}^q \Big] \tau.
\end{align}

By the standard estimation of spectral Gakerin approximation that
$\|({\rm Id}_H-\PP_N) u\|\le C N^{-\gamma} \|u\|_\gamma$ for any $u\in \dot H^\gamma$ and the Sobolev embedding that
$\dot H^{1/2-1/q}\hookrightarrow L_x^q$ for $q\ge 2$, we obtain
\begin{align*}
\sup_{m\in \zz_M} \ee\Big[ \|({\rm Id}_H-\PP_N) u(t_m)\|^2\Big]
\le C N^{-2\gamma} \sup_{t\in [0,T]} \ee\Big[\|u(t)\|^2_\gamma \Big],
\end{align*}
and
\begin{align*}
& \sum_{m\in \zz_M} \ee\Big[\|({\rm Id}_H-\PP_N) u(t_m)\|_{L_x^q}^q \Big] \tau
\nonumber  \\
&\le \sup_{t\in [0,T]}  \ee\Big[\|({\rm Id}_H-\PP_N) (-A)^{\frac12(\frac12-\frac1q)} u(t)\|^q \Big] T \\
&\le C N^{-q(\widetilde \gamma-\frac12)-1}
\sup_{t\in [0,T]} \ee\Big[\|u(t)\|^q_{\widetilde \gamma} \Big].
\end{align*}
for any $\gamma,\widetilde \gamma\in (0,1/2)$.
In particular, for $\gamma\in (0,1/2)$ one can choose
$$\widetilde \gamma=\frac12-\frac{1-2\gamma}q \in \Big(0,\frac12\Big)$$
and get
\begin{align*}
&\sup_{m\in \zz_M} \ee\Big[ \|({\rm Id}_H-\PP_N) u(t_m)\|^2\Big]
+\sum_{m\in \zz_M} \ee\Big[\|({\rm Id}_H-\PP_N) u(t_m)\|_{L_x^q}^q \Big] \tau \\
&\le C N^{-2\gamma}
\Big(\sup_{t\in [0,T]} \ee\Big[\|u(t)\|^2_\gamma \Big]
+\sup_{t\in [0,T]} \ee\Big[\|u(t)\|^q_{\widetilde \gamma} \Big]\Big) \\
&\le C N^{-2\gamma},
\quad \forall\  \gamma\in \Big(0,\frac12\Big).
\end{align*}
In terms of \eqref{x-xnm01} and the above estimation, to show the estimations \eqref{u-unm0} we only need to prove
\begin{align} \label{en}
\sup_{m\in \zz_M}\ee\Big[\|e_N^m\|^2\Big]
+\sum_{m\in \zz_M} \ee\Big[\|e_N^m \|_{L_x^q}^q \Big] \tau
\le C \Big(N^{-2\gamma}+\tau^{1/2} \Big),
\quad \forall\ \gamma\in \Big(0,\frac12\Big).
\end{align}

Subtracting \eqref{z-var} from \eqref{full} with $v=v_N=e_N^{m+1}\in V_N\subset \dot H^1$, we get
\begin{align}\label{err-equ}
&\<({\rm Id}_H-\PP_N) (z(t_{m+1}-z(t_m))), e_N^{m+1}\>
+ \<e_N^{m+1}-e_N^m, e_N^{m+1}\>  \nonumber \\
&=-\int_{t_m}^{t_{m+1}} \<\nabla (z-z_N^{m+1}), \nabla e_N^{m+1}\>  {\rm d}r
+\int_{t_m}^{t_{m+1}} \<F(u)-F_N^{m+1}, e_N^{m+1}\>  {\rm d}r.
\end{align}
Since $\PP_N$ is an $L^2$-projection, we have
$$\ee\Big[\<({\rm Id}_H-\PP_N) (z(t_{m+1}-z(t_m))), e_N^{m+1}\>\Big]=0.$$
By the elementary identity $(a-b)a=\frac12 (a^2-b^2)+\frac12 (a-b)^2$, we get
\begin{align} \label{est-e}
\ee\Big[\<e_N^{m+1}-e_N^m, e_N^{m+1}\>\Big]
=\frac12 \Big(\ee\Big[\|e_N^{m+1}\|^2\Big]
-\ee\Big[\|e_N^m\|^2\Big]\Big)
+\frac12 \ee\Big[\|e_N^{m+1}-e_N^m\|^2\Big].
\end{align}
Applying the fact that $\<\nabla ({\rm Id}_H-\PP_N) u, \nabla v_N\>=0$ for any $u\in \dot H^1$ and $v_N\in V_N$, Cauchy--Schwarz inequality and the estimation \eqref{hol-y} with $\beta=1/2$, we obtain
\begin{align} \label{est-y}
&\ee\bigg[-\int_{t_m}^{t_{m+1}} \<\nabla (z(r)-z_N^{m+1}), \nabla e_N^{m+1}\> {\rm d}r\bigg]  \nonumber \\
&=-\int_{t_m}^{t_{m+1}}
\ee\Big[\<\nabla (z(r)-z(t_{m+1})),
\nabla e_N^{m+1}\>\Big] {\rm d}r
-\ee\Big[\|\nabla e_N^{m+1}\|^2 \Big]\tau \nonumber \\
&\le \frac12 \int_{t_m}^{t_{m+1}}
\ee\Big[\|\nabla (z(r)-z(t_{m+1}))\|^2\Big] {\rm d}r
-\frac12 \ee\Big[\|\nabla e_N^{m+1}\|^2 \Big]\tau \nonumber \\
&\le C \tau^{3/2}
-\frac12 \ee\Big[\|\nabla e_N^{m+1}\|^2 \Big]\tau.
\end{align}

For the third term in Eq. \eqref{err-equ}, the monotone condition \eqref{con-f} of $f$, H\"older and Young inequalities and the relation \eqref{unm} imply that
\begin{align*}
& \ee\bigg[\int_{t_m}^{t_{m+1}} \<F(u(r))-F(u_N^{m+1}), e_N^{m+1}\> {\rm d}r\bigg] \\
&=\int_{t_m}^{t_{m+1}} \ee \Big[\<F(u(r))-F(u(t_{m+1})), e_N^{m+1}\> {\rm d}r \Big] \\
&\quad +\ee \Big[\<F(u(t_{m+1}))-F(\PP_N u(t_{m+1})), e_N^{m+1}\> \Big] \tau \\
&\quad +\ee \Big[\<F(\PP_N u(t_{m+1}))-F(u_N^{m+1}), e_N^{m+1}\>\Big] \tau \\
&\le \frac C{\zeta} \int_{t_m}^{t_{m+1}}
\ee\Big[\|F(u(r))-F(u(t_{m+1}))\|^2\Big] {\rm d}r \\
&\quad +\frac C{\zeta}
\ee\Big[\|F(u(t_{m+1}))-F(\PP_N u(t_{m+1}))\|^2\Big] \tau \\
&\quad +(b+\zeta) \ee\Big[\|e_N^{m+1}\|^2 \Big] \tau
-L_f \ee\Big[\|e_N^{m+1}\|_{L_x^q}^q \Big] \tau,
\end{align*}
where $\zeta$ is an arbitrary positive number.
By the estimation \eqref{hol-f} with $\beta=1/2$, we get
\begin{align*}
\frac C{\zeta} \int_{t_m}^{t_{m+1}}
\ee\Big[\|F(u(r))-F(u(t_{m+1}))\|^2\Big] {\rm d}r
\le \frac C{\zeta} \tau^{3/2}.
\end{align*}
By the condition \eqref{con-f'} and the moments' estimations \eqref{bou} and \eqref{bou-pnx}, we have
\begin{align*}
\frac C{\zeta}
& \ee\Big[\|F(u(t_{m+1}))-F(\PP_N u(t_{m+1}))\|^2\Big] \tau \\
& \le \frac C{\zeta}
\Big[1+\Big(\ee\Big[\|u(t_{m+1})\|_{L_x^\infty}^{4(q-2)}\Big]\Big)^\frac12
+\Big(\ee\Big[\|\PP_N u(t_{m+1})\|_{L_x^\infty}^{4(q-2)}\Big]\Big)^\frac12 \Big]  \\
&\quad \times \Big(\ee\Big[\|({\rm Id}_H-\PP_N) u(t_{m+1})\|^4 \Big]\Big)^\frac12 \tau
\le \frac C{\zeta} N^{-2\gamma} \tau.
\end{align*}
Consequently,
\begin{align} \label{est-f}
& \ee\bigg[\int_{t_m}^{t_{m+1}} \<F(u(r))-F(u_N^{m+1}), e_N^{m+1}\> {\rm d}r \bigg] \nonumber \\
&\le \frac C \zeta \Big(N^{-2\gamma} +\tau^{1/2} \Big) \tau
+(b+\zeta) \ee\Big[\|e_N^{m+1}\|^2 \Big] \tau
-L_f \ee\Big[\|e_N^{m+1}\|_{L_x^q}^q \Big] \tau.
\end{align}
Combining the above estimations \eqref{est-e}--\eqref{est-f}, we derive
\begin{align*}
&\frac12 \Big(\ee\Big[\|e_N^{m+1}\|^2\Big]
-\ee\Big[\|e_N^m\|^2\Big]\Big)
+\frac12 \ee\Big[\|\nabla e_N^{m+1}\|^2 \Big]\tau \\
&\le \Big(C+\frac C \zeta\Big)
\Big(N^{-2\gamma}+\tau^{1/2} \Big) \tau
+(b+\zeta) \ee\Big[\|e_N^{m+1}\|^2 \Big] \tau
-L_f \ee\Big[\|e_N^{m+1}\|_{L_x^q}^q \Big] \tau.
\end{align*}
Then we deduce that
\begin{align*}
& \Big(1-2(b+\zeta) \tau\Big) \ee\Big[\|e_N^{m+1}\|^2\Big]
+\ee\Big[\|\nabla e_N^{m+1}\|^2 \Big]\tau
+2 L_f \ee\Big[\|e_N^{m+1}\|_{L_x^q}^q \Big] \tau \\
&\le \ee\Big[\|e_N^m\|^2\Big]
+\Big(C+\frac C \zeta\Big)
\Big(N^{-2\gamma}+\tau^{1/2} \Big) \tau.
\end{align*}

Summing over $m=0,1,\cdots,l-1$ with $1\le l\le M$, we obtain
\begin{align*}
& \Big(1-2(b+\zeta) \tau\Big)  \ee\Big[\|e_N^l\|^2\Big]
+\sum_{m=0}^l \ee\Big[\|\nabla e_N^m \|^2 \Big]\tau
+2 L_f \sum_{m=0}^l \ee\Big[\|e_N^m\|_{L_x^q}^q \Big] \tau \\
&\le \Big(C+\frac C \zeta\Big)
\Big(N^{-2\gamma}+\tau^{1/2} \Big) 
+2(b+\zeta) \sum_{m=0}^{l-1}\ee\Big[\|e_N^m\|^2\Big] \tau.
\end{align*}
When $b<0$ we set $\zeta=-b$ and $\tau\in (0,1)$, while when $b>0$ we set $\tau<1/(4b)$ and  $\zeta$ sufficiently small.
Through the discrete Gr\"onwall inequality, we conclude the estimation \eqref{en}.
This completes the proof of \eqref{u-unm0}.
\end{proof}

\section{Numerical Experiments}

In this section, we give several numerical tests to verify the optimality of the strong convergence rate under the $l_t^\infty L^2_\omega L^2_x \cap l_t^q L^q_\omega L^q_x$-norm in Theorem \ref{u-unm} for the backward Euler--spectral Galerkin scheme \eqref{full}-\eqref{unm}.

Due to Lemma \ref{ou-err} and Remark \ref{rk-ou-err}, the spatial convergence rate of the backward Euler--spectral Galerkin scheme \eqref{full}-\eqref{unm} is sharp.
Our main concern here is to simulate the temporal strong convergence rate, under the $l_t^\infty L^2_\omega L^2_x \cap l_t^q L^q_\omega L^q_x$-norm (with $q=6$), of the fully discrete scheme \eqref{full} for the following SPDE driven by an additive Brownian sheet $W$:
\begin{align}
\frac{\partial u}{\partial t}
&=\frac{\partial^2 u}{\partial x^2}
+\big(u^4-u^5 \big)
+\frac{\partial^2 W}{\partial t \partial x},  \label{sac2}
\end{align}
with homogeneous Dirichlet boundary condition \eqref{dbc} and the initial value
\begin{align*}
u_0(x)=\sum_{k=1}^\infty \frac{e_k(x)}{k^2},
\quad e_k(x)=\sqrt 2 \sin(k\pi x),\quad x\in (0,1).
\end{align*}
% One can show that $u_0\in \dot H^\gamma$ for any $\gamma\in (0,3/2)$.

We use the backward Euler--spectral Galerkin scheme \eqref{full}-\eqref{unm} with $f(x)=x^4-x^5$ and the initial datum
$z_0^N=\PP_N u_0=\sum_{k=1}^N k^{-2} e_k$
to fully discretize Eq. \eqref{sac2}.
To simulate the approximate Ornstein--Uhlenbeck process $W_A^N$,
it is clear that
\begin{align*}
W_A^N(t_m)
=\int_0^{t_m} S(t_m-r) \PP_N {\rm d}W_H(r)
=\sum_{k=1}^N \bigg[ \int_0^{t_m} e^{-\lambda_k (t_m-r)} {\rm d}\beta_k(r) \bigg] e_k,
\end{align*}
where
$$\bigg\{\int_0^{t_m} e^{-\lambda_k (t_m-r)} {\rm d}\beta_k(r)\sim \mathcal N \Big(0,\frac{1-e^{-2\lambda_k t_m}}{2\lambda_k} \Big):
\ m\in \zz_M \bigg\}$$
is a sequence of independent centered Gaussian random variable.
Thus
\begin{align*}
W_A^N(t_m)
=\sum_{k=1}^N \sqrt{\frac{1-e^{-2\lambda_k t_m}}{2\lambda_k}} \zeta_k e_k,
\end{align*}
where $\{\zeta_k\}_{k\in \zz_N}$ is a sequence of independent normally distributed random variables.

To simulate a reference solution, we perform the full discretization by $N=512$ for the dimension of spectral Galerkin approximation and by
$\tau=2^{-13}$ for the temporal step size of the scheme \eqref{full}.
The expectation is approximated from the average of $1000$ sample paths.
To simulate the temporal strong convergence rate of the scheme \eqref{full}, we take the step size by $\tau=2^{-i}$ with $i=7,8,9,10$.

Figure \ref{fig} displays the temporal mean-square convergence rate (under the $l_t^\infty L^2_\omega L^2_x$-norm) and another type of temporal strong convergence rate under the $l_t^6 L^6_\omega L^6_x$-norm of the backward Euler--spectral Galerkin scheme \eqref{full}-\eqref{unm} for Eq. \eqref{sac2}.
By Theorem \ref{u-unm}, the strong convergence orders under the $l_t^\infty L^2_\omega L^2_x$-norm and the $l_t^6 L^6_\omega L^6_x$-norm are $1/4$ and $1/{2q}=1/{12}$, respectively.
The temporal mean-square convergence rate $\OO(\tau^{1/4})$ of the scheme \eqref{full}-\eqref{unm} can be confirmed in Figure \ref{fig} (a), and
the temporal convergence rate $\OO(\tau^{1/12})$ of the scheme \eqref{full}-\eqref{unm} can be confirmed in Figure \ref{fig} (b).

\begin{figure}[htbp]
\centering
\subfigure[]{
\begin{minipage}{6cm}
\centering
\includegraphics[height=5.5cm,width=6cm]{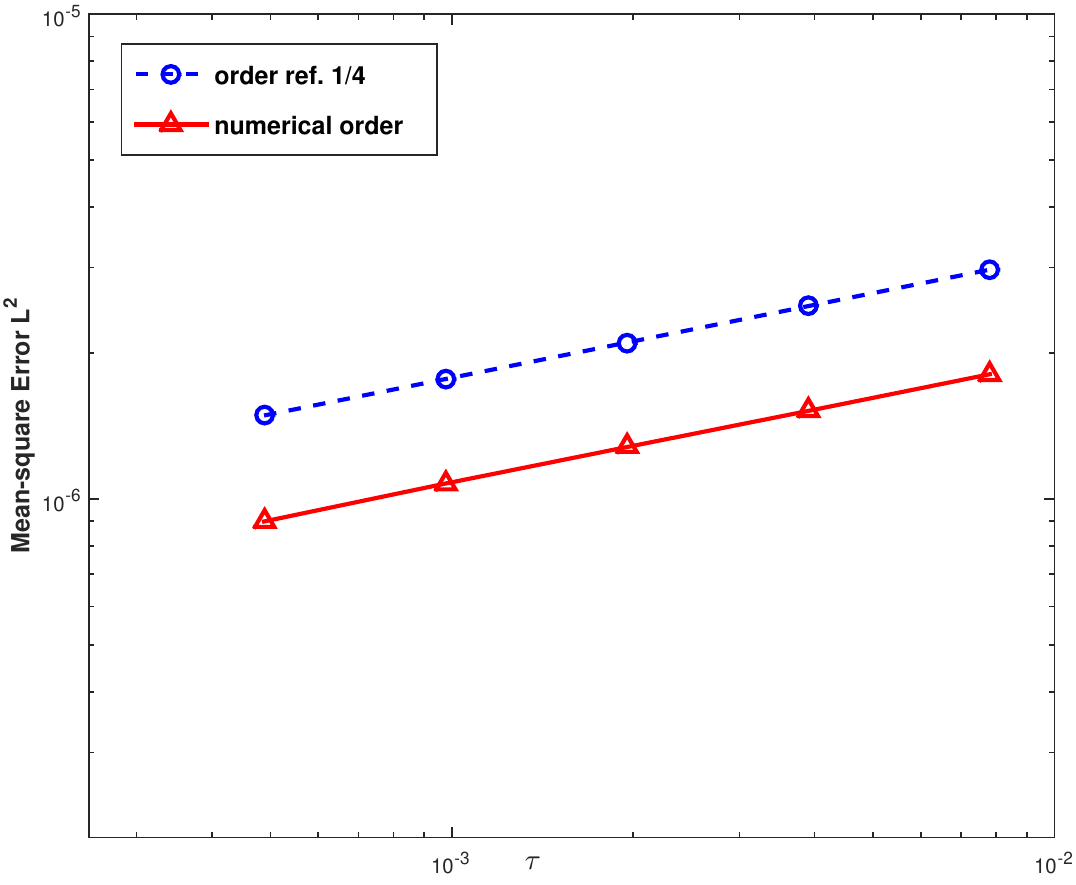}\\
\end{minipage}
}
\subfigure[]{
\begin{minipage}{6cm}
\centering
\includegraphics[height=5.5cm,width=6cm]{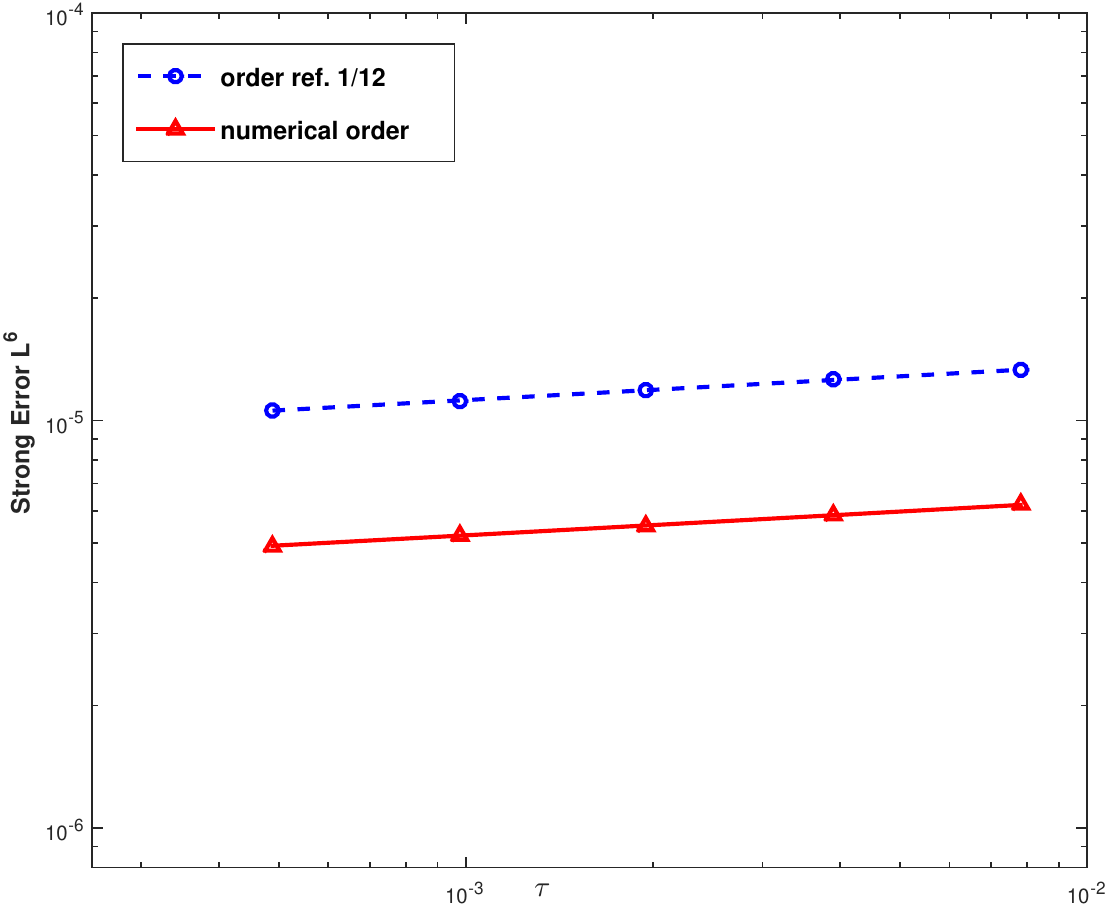}\\
\end{minipage}
}
\caption{Temporal convergence rates under the norms of (a) $l_t^\infty L^2_\omega L^2_x$ and (b) $l_t^6 L^6_\omega L^6_x$.}  \label{fig}
\end{figure}

\section*{Acknowledgements}

We thank the anonymous referee for very helpful remarks and suggestions.
We also thank Dr. Lihai Ji from Institute of Applied Physics and Computational Mathematics in  Beijing for his help and comments on numerical tests.
This work is partially supported by Hong Kong Research Grants Council General Research Fund (grants 15300417 and 15325816).

\bibliographystyle{plain}
\bibliography{bib}
\end{document}